\documentclass[12pt]{article}
\usepackage[latin1]{inputenc}
\usepackage{amsmath}
\usepackage{enumerate}
\usepackage{amsfonts}
\usepackage{amssymb}
\usepackage{hyperref}

\newcommand{\qed}{\hfill $\Box$ \vspace{\baselineskip}}

\numberwithin{equation}{section}

\newtheorem{teo}{Theorem}[section]

\newtheorem{lem}[teo]{Lemma}

\newtheorem{es}[teo]{Exemple}

\newcommand{\be}{\begin{equation}}
\newcommand{\ee}{\end{equation}}

\def\no{\noindent}

\def\ba{\begin{array}}
\def\ea{\end{array}}

\def\lb{\lambda}

\def\P{\mathsf{P}}
\def\E{\mathsf{E}}

\def\R{\mathbb{R}}

\def\vp{\varphi}
\def\ve{\varepsilon}

\begin{document}

\begin{center}
{\huge A very short proof of the subgaussian convex-order comparison theorem}
\end{center}

\medskip

\begin{center}
Antoine Song\footnote{California Institute of Technology, 177 Linde Hall, 1200 E. California Blvd., Pasadena, CA
91125.
Email address: aysong@caltech.edu}and
 Michel Talagrand\footnote {Paris, France, michel.talagrand@gmail.com}
\end{center}

\vspace{1cm}

\begin{abstract} We provide a very short proof of the following comparison theorem, which has recently played an important role: every centered subgaussian random vector is dominated in convex order, up to a universal constant, by a standard Gaussian random vector.

\end{abstract}
We denote by $\langle \cdot,\cdot\rangle$ (resp. $\Vert \cdot\Vert$) the standard scalar product (resp. the standard norm) on $\R^N$, and by $G$ a standard Gaussian vector in $\R^N$. We say that a random vector $X$ in $\R^N$ is {\it subgaussian} if for each $a\in \R^N$, we have
$$\E \exp (\langle a,X\rangle )\leq \E \exp (\langle a ,G\rangle ) = \exp (\Vert a\Vert^2 /2).$$
In particular, in this note subgaussian random vectors are always centered.

Ramon van Handel recently showed in \cite{v} that the following theorem could be deduced from a general comparison result of Jingbo Liu \cite[Corollary 2]{Liu25}:

\begin{teo} [Subgaussian convex-order comparison] \label{main}
There exists a universal constant $L$ such that for any subgaussian random vector $X$ in $\R^N$ and any convex function $\vp$ with $\E\vp(G)<\infty$, we have
$$\E\vp(X/L)\leq \E \vp(G).$$
\end{teo}

\textbf{Historical notes:} A first understanding of general subgaussian processes was obtained in \cite{t} by the second-named author, where it is proved that for some universal constant $L>0$, for every semi-norm $p(\cdot)$ and every  subgaussian random vector $X$ in $\R^N$, we have $\E p(X)\leq L\E p(G)$ (see also \cite[Theorem 2.10.11]{t14}).
This result was revisited with different arguments in \cite{vh1,vh2}.
Recently, Jingbo Liu discovered a strengthening of this comparison theorem \cite[Corollary 2]{Liu25} with a new and elementary approach, based on a ``lifting argument'' and older work of Fernique on stationary Gaussian processes \cite{f}.
Ramon van Handel observed in \cite{v} that \cite[Corollary 2]{Liu25} quickly leads to Theorem \ref{main}. In \cite[Subsection 2.4.2]{Son26}, the first-named author sketched a different short proof of Theorem \ref{main} by combining \cite{t} and a simple ``tensor power trick'', which gives
the geometric interpretation that  subgaussian vectors essentially  form the convex hull of Gaussian vectors up to universal scaling constants. This inspired the even shorter proof of this note\footnote{A.I. was not used for writing this paper apart from bibliographic search.}, worked out by the second-named author. Theorem \ref{main} played an important role in the resolution of the convexity conjecture of the second-named author \cite{t95,t10}  by Dongming (Merrick) Hua, Stefan Tudose and the first-named author \cite{HST26}. Sharp constants for Theorem \ref{main} in dimension 1 were computed in \cite{dp,zhang}.

\vspace{1em}

\no{\bf Proof of Theorem \ref{main}.}

\begin{lem}\label{lemma}
There exists a universal constant $L$ with the following property.  Consider a  subgaussian random vector $X$ in $\R^N$, standard Gaussian vector $G$ in $\R^N$ and a convex set $C$ in $\R^N$ with $\P(G\in C)\geq 7/8$. Then $\P(X\in LC)\geq 1/2$.
\end{lem}
\no {\bf Proof:}	Let $C'=C\cap(-C)$ so that $\P(G\in C')\geq 3/4.$ Define the semi-norm $p(x)=\inf\{ \lb ; x\in \lb C'\}$. By the standard Borel lemma \cite[Theorem 2.7]{ver}, $\E \Vert G\Vert\leq L'$ where $L'$ is a universal constant. By \cite{t} (see also \cite[Theorem 2.10.11]{t14}), $\E p(X)\leq L''$ for some positive universal constant $L''$ and so 
 $\P(p(X) \leq 2 L'')\geq 1/2$, which means that  $\P(X \in 3 L''C')\geq 1/2$. \qed

To prove Theorem \ref{main}, we can assume that $\vp$ is the maximum of finitely many affine functions on $\R^N$, because any convex function on $\R^N$ is a limit of an increasing sequence of functions of this form.
Consider a large integer $M$ and i.i.d copies $X_i$ of $X$ and $G_i$ of $G$. The basic observation is that the random vector $\bar X=(X_1,\ldots, X_M)$ in $\R^{NM}$ is subgaussian and the random vector $\bar G=(G_1, \ldots, G_M)$ is standard Gaussian. 
Denote by $\bar x=(x_1, \ldots, x_M)$ the generic point of $\R^{NM}$ where $x_i\in \R^N$.  The function $\bar \vp$ on $\R^{NM}$ defined by 
$$\bar \vp(\bar x)=\frac{1}{M}\sum_{i\leq M} \vp(x_i)$$
is convex. Consider an arbitrary $\ve>0$. The law of large numbers shows that for $M$ large enough, the convex set $C=\{\bar \vp\leq \E \vp(G)+\ve\}$ satifies $\P(\bar G\in C)\geq 7/8$. Then by Lemma \ref{lemma}, $\P(\bar X\in LC)\geq 1/2$ for some universal constant $L>1$. Besides, note that $\vp(X/L)$ is integrable since $\vp$ is the maximum of finitely many affine functions. Therefore, if $M$ is large enough, by the law of large numbers again, we can find some $\bar x \in LC$ such that $\bar \vp(\bar x/L)\geq \E\vp(X/L)-\ve$. Hence
 $$\E\vp(X/L)-\ve \leq \vp(\bar x/L)\leq \E\vp(G)+\ve $$
 which concludes the proof since $\ve$ is arbitrarily small. \qed

\end{document}